\documentclass{article}

\usepackage[english]{babel}

\usepackage[utf8]{inputenc}   
\usepackage{amsmath}
\usepackage{amsfonts}
\usepackage{amssymb}
\usepackage{amsthm}

\newcommand{\cat}{\cdot}
\renewcommand{\a}{\alpha}
\renewcommand{\b}{\beta}
\newcommand{\eps}{\varepsilon}
\newcommand{\f}{\varphi}
\renewcommand{\l}{\lambda}
\newcommand{\R}{\mathbb{R}}
\newcommand{\Q}{\mathbb{Q}}
\newcommand{\N}{\mathbb{N}}

\newcommand{\ov}{\overline}
\renewcommand{\r}{\ov\R}
\renewcommand{\Im}{\mathrm{Im}\,}

\newtheorem{lem}{Lemma}
\newtheorem{prop}{Proposition}

\title{Iteration of Exponentials with Sign Changes}

\author{
Pierre Mazet \thanks{piermazet@laposte.net}
\and
Emmanuel Halberstadt \thanks{Emmanuel Halberstadt passed away on September 22, 2021 and unfortunately could not participate in the final version of this article}
}

\date{}

\begin{document}

\maketitle

\abstract{
In this paper we consider the iteration of infinitely many signed exponentials with the same base but the signs may vary. We show that for every base in an explicit interval this iteration converges for any sequence of signs and all the real numbers are possible limits. We give some more results for a base outside this interval.

\bigskip

Dans cet article, nous considérons l’itération d’une infinité d’exponen\-tielles signées avec la même base mais les signes peuvent varier. Nous montrons que pour toute base dans un intervalle explicite, cette itération converge pour toute suite de signes et que tous les nombres réels sont des limites possibles. Nous donnons quelques résultats supplémentaires pour une base en dehors de cet intervalle.
}

\section{Introduction.}
\subsection*{The initial problem}
More than twenty years ago, our colleague Michel Lazarus\footnote{Michel Lazarus died on September 15, 2010, he was a lecturer at the Pierre et Marie Curie University (Paris VI).} 
asked the first author
the two questions below (we don't know if he had the answers and the origin of this problem). 
Given an infinite sequence $\eps=(\eps_1,\eps_2,\eps_3,\ldots)$ of $+$ and $-$ signs we form the following expression 
\[(E):\hskip1cm 
\eps_1e^{\textstyle \,\,\eps_2\,e^{\textstyle \,\,\eps_3\,\,e^{\textstyle\,\,
\eps_4\,e^{\ \textstyle \cdot{^{\ \ \textstyle \cdot^{\ \ \textstyle \cdot}}}}}}}
\]
Michel Lazarus then asked the following questions:
\begin{itemize}
\item[]$Q_1$: For a given $\eps$, does the expression $(E)$ represent an element of $\r$?
\item[] $Q_2$: For $t$ in the extended real line $\r$, can we represent $t$ by a suitable sequence  $\eps$?
\end{itemize}
\medskip
To clarify question $Q_1$ we introduce for each $n\in \N^*$ the expression
\begin{align}
  u_{n,\eps}=\eps_1\exp\bigg(\eps_2\exp\Big(\dots \big(\eps_n e\big)\Big)\bigg) \label{eq:u_n_eps}
\end{align}
obtained by truncating the expression $(E)$ at the $n$-th exponential.
Question $Q_1$ is then rewritten:

\begin{itemize}
\item[] $Q_1$: For given $\eps$, does the sequence $(u_{n,\eps})_n$ have a limit in $\r$?
\end{itemize}

\medskip
In section \ref{suitable} we prove  that the answer to these two questions is always positive.

\subsection*{The extended problem}
We also study the analogues of questions $Q_1$ and $Q_2$ when we replace $e$ by $e^a$ for some $a>0$ and we show that the answers depend now on the value of $a$.
In the following, we fix a real number $a>0$ and we place ourselves within the framework of this extended problem that we specify below.

\medskip
Let $f_+$ and $f_-$ denote the applications of $\r$ into itself defined by
\[f_+(x) = e^{ax} \qquad f_-(x) = -e^{ax}
  \]
with the usual conventions $e^{-\infty}=0$ and $e^{+\infty}=+\infty$.
For a sequence $\eps$ of $+$ and $-$ and an integer $n$ we denote by $f_{n,\eps}$ the composition
\[f_{n,\eps}:=f_{\eps_1}\circ f_{\eps_2}\circ \dots \circ f_{\eps_n}.
\]
In particular $f_{0,\eps}$ is the identical map. We then generalize expression~\eqref{eq:u_n_eps} by writing  $u_{n,\eps}=f_{n,\eps}(1)$ even when $a\neq 1$, omitting the dependence on $a$.

If the sequence $(u_{n,\eps})_n$ has a limit $t$ in $\r$ we say that $\eps$ {\bf represents} $t$. We say that an element $t$ of $\r$ is {\bf representable} if there exists (at least) a sequence $\eps$ that represents it.
Finally we say that $a$ is {\bf suitable} if any sequence $\eps$ represents an element of $\r$ and any element of $\r$ is representable, in other words if the answer to questions $Q_1$ and $Q_2$ is always positive.

\subsection*{Note.}
Much work has been done around iterated exponentials, See for instance ~\cite{EUL},~\cite{BAK1},~\cite{BAK2},~\cite{MIS},~\cite{SHELL},~\cite{THRON}. But we never found any article concerning the problem studied here.

\subsection*{Our results}

\begin{prop}\label{un}
The set of sequences  $\eps$ for which the sequence $(u_{n,\eps})_n$ does not converge (i.e.  such that $\eps$ does not represent any element of $\r$) is at most countable. The set of representable elements is therefore not empty.
\end{prop}

\begin{prop}\label{deux}
Suppose $a\le 1/e$. Then, any sequence $\eps$ represents an element of $\r$ (i.e., the sequence of $(u_{n,\eps})_n$ converges). Moreover, the set of representable elements of $\r$ is a closed set with empty interior in $\r$; $a$ is therefore not suitable.
\end{prop}
In the proof of Proposition~\ref{deux} we give a more precise description of the set of representable elements.

\begin{prop}\label{trois}
When $a>e$ there are sequences $\eps$ for which the sequence $u_{n,\eps}$ converges and others for which it diverges.

Likewise there are elements of $\r$ which are representable and others which are not.
In particular $a$ is not suitable.
\end{prop}
We could not get much more specific.

\begin{prop}\label{quatre}
The equation $(x+1)x^{-1/(x+1)}=e$ has a unique solution between $0$ and $1$ that we denote by $A$.

When $A<a\le e$ any sequence $\eps$ represents an element of $\r$ and any element of $\r$ is representable; $a$ is therefore suitable. In particular, this is the case for the initial problem.
\end{prop}

We have $A\sim 0.394$. Since $1/e$ is approximately $0.368$,  there remains a small range of values of $a$ for which we have not been able to solve the problem but it seems very likely to us that for these values $a$ is still suitable.

\section{Our method and first results}

\medskip
We identify a sequence $\eps=(\eps_1,\eps_2,\eps_3,\ldots)$ with the infinite word $\eps_1\,\eps_2\,\eps_3\,\ldots$ formed with the letters $+$ and
$-$.
More generally, we consider words, finite or not, over the alphabet $\{+,-\}$. The concatenation of a finite word $\a$ and a finite or infinite word $\b$ is denoted $\a\cat\b$ .

If $\a=\a_1\a_2\dots$ is an infinite word, we recall from the previous section the notation
\[f_{n,\a}=f_{\a_1}\circ f_{\a_2}\circ \dots \circ f_{\a_n}\ .
\]
We extend this notation to the case where $\a$ is a finite word of length at least $n$ and, if $\a$ has length $n$, we simply write
$f_\a$ for $f_{n,\a}$.
(Of course, if $n=0$, $f_{0,\a}$ is the identity on $\r$.)
Note that $f_{n,\a}$ is continuous and strictly monotone. It is increasing if $\a$ has an even number of $-$ signs among the $n$ first signs, and decreasing otherwise.
The extended problem is therefore the study of the convergence of the sequence of $u_{n,\eps}=f_{n,\eps}(1)$ and of the possible limits of these sequences.
This leads us to study the image of $f_{n,\eps}$. This image is a closed interval $I(n,\eps)=[x_{n,\eps},y_{n,\eps}]$ where
$x_{n,\eps}$ is the image of $-\infty$ (resp. $+\infty$) and $y_{n,\eps}$ is the image of $+\infty$ (resp . $-\infty$) if $f_{n,\eps}$ is increasing (resp. decreasing).

Note that we have $f_{n+1,\eps}=f_{n,\eps}\circ f_{\eps_{n+1}}$, which ensures the inclusion of $I(n+ 1,\eps)$ in $I(n,\eps)$. We thus have 
$x_{n,\eps}\leq x_{n+1,\eps}\leq y_{n+1,\eps}\leq y_{n,\eps}$. Hence, for $\eps$ a given infinite word, the sequence of $x_{n,\eps}$ is increasing and therefore has a limit $x(\eps)$, possibly infinite, that is the supremum in $\r$ of the $x_{n,\eps}$.
Similarly, $(y_{n,\eps})_n$ is a decreasing sequence whose limit  $y(\eps)$ is the infimum in $\r$ of the $y_{n,\eps}$.
We obviously have $x(\eps)\le y(\eps)$.
In other words the intervals $I(n,\eps)$ form a decreasing sequence of closed intervals whose intersection is the interval $[x(\eps),y(\eps)]$.

We set 
$I(\eps)=[x(\eps),y(\eps)]$, $I^o(\eps) =]x(\eps),y(\eps)[$ and  $I^o(n,\eps)=]x_{n,\eps},y_{n,\eps}[$.
If $I(\eps)$ is reduced to a singleton (i.e. $x(\eps)=y(\eps)$) then $I^o(\eps)$ is empty.

\smallskip\noindent
{\bf Note.}\quad The intersection $J$ of the $I^o(n,\eps)$ is not the interval $I^o(\eps)$. More precisely we can show that this only happens when $a>1/e$ and the word $\eps$ contains only a finite number of $-$ signs, in which case $J$ and $I^o( \eps)$ are empty. In the other cases, $x(\eps)$ or $y(\eps)$ is an extremity of $J$, and more frequently both.

\smallskip
If $\gamma\cat\eps$ is concatenation of a finite word $\gamma$ of length $k$
with $\eps$ we observe that
 $f_{k+n,\gamma\cat\eps}=f_\gamma\circ f_{n,\eps}$ and thus
$$I(k+n,\gamma\cat\eps)=f_\gamma\big(I(n,\eps)\big)\ ,\ I(\gamma\cat\eps)=f_\gamma\big(I(\eps)\big)\ ,\ I^o(\gamma\cat\eps)=f_\gamma\big(I^o(\eps)\big)\ .$$

\smallskip
Since, for any infinite word $\eps$ we have $x_{n,\eps}\le u_{n,\eps}\le y_{n,\eps}$, we conclude:
\begin{lem}\label{l0}
Let $\eps$ be an infinite word of signs.
If $\eps$ represents an element $t$ of $\r$, then $t$ belongs to the interval $I(\eps)$.
If $I(\eps)$ is reduced to a single element (i.e. $x(\eps)=y(\eps)$) then $\eps$ represents this element.
\end{lem}

\begin{lem}\label{l1}
Let $\eps$ and $\eps'$ be two distinct infinite words. Then, for $n$ sufficiently large, the intersection $I(n,\eps)\cap I(n,\eps')$ is either empty or reduced to an element which is an extremity of each of these two intervals. The same holds for the intersection $I(\eps)\cap I(\eps')$.
Therefore, for $n$ large enough $I^o(n,\eps)\cap I(n,\eps')$ is empty as well as $I^o(\eps)\cap I(\eps') $.
\end{lem}

\begin{proof}
Let $n_0$ denote the first index such that $\eps_{n_0}\not=\eps'_{n_0}$. By swapping $\eps$ and $\eps'$ if necessary, we can assume $\eps_{n_0}=+$ and $\eps'_{n_0}=-$ and
 write
$\eps=\a\cat+\cat\beta$ and $\eps'=\a\cat-\cat\beta'$, where $\a$ is a word of length $n_0-1$ (possibly empty if $ n_0=1$). We then have $I(n_0,\eps)=f_\a([0,+\infty])$ and
$I(n_0,\eps')=f_\a([-\infty,0])$. As $f_\a$ is monotone we conclude $I(n_0,\eps)\cap I(n_0,\eps')= \{f_\a(0)\}$; this intersection is indeed reduced to a single element which is an extremity of each of the intervals.

Since the interval $I(n_0,\eps)$ (resp. $I(n_0,\eps'))$ contains the interval $I(\eps)$ (resp. $I(\eps')$) and each of the intervals $I(n,\eps)$ (resp. $I(n,\eps')$) for $n\ge n_0$, the lemma follows.
\end{proof}

\begin{lem}\label{l1'}
Let $\eps$ be an infinite word of signs. Any element of $I^o(\eps)$ represented by an infinite word must be represented by $\eps$. Therefore, if
$I(\eps)$ is not reduced to a single element, then $I^o(\eps)$ contains at most one representable element and infinitely many non-representable elements.
\end{lem}

\begin{proof}
Let $t\in I^o(\eps)$. If $t$ is represented by an infinite word $\eps'$, then $t\in I^o(\eps)\cap I(\eps')$. This intersection is thus nonempty and the previous lemma ensures $\eps'=\eps$.
\end{proof}

\noindent{\bf Remark.}\quad This lemma shows that if the answer to question $Q_2$ is always positive (i.e. any element of $\r$ is representable) then all the $I(\eps)$ are reduced to a single element and Lemma~\ref{l0} ensures that the answer to question $Q_1$ is also always positive
(i.e. any infinite word $\eps$ represents an element of $\r$).

\begin{lem}\label{l2}
For every $t$ in $\r$ there is at least one and at most two infinite words $\eps$ such that $t\in I(\eps)$.
\end{lem}
\begin{proof}
Suppose that $t\in\Im(f_\a)$ for some word $\a$ of length $n$. We therefore have $t=f_\a(u)$ with $u$ in $\r$. If $u\in\Im(f_+)$ (i.e. $u\ge0$) we can write
$u=f_+(v)$ and therefore $t=f_{\a'}(v)$ with $\a'=\a\cat+$ . Similarly, if $u\le0$, we can write $t=f_{\a'}(v')$ with $\a'=\a\cat-$ . Thus $t\in\Im(f_{\a'})$ where $\a'$ is an extension of $\a$ of length $n+1$.
This allows to build by induction (from the empty word for which $\Im(f_\emptyset)=\r$) an infinite word $\eps$ for which $t$ belongs to all the $I(n,\eps )$ and therefore to $I(\eps)$.

Suppose then that $t$ belongs to $I(\eps)$, $I(\eps')$ and $I(\eps'')$ for three different infinite words. Then, for all $n$, $t$ belongs to each of the intervals $I(n,\eps)$, $I(n,\eps')$ and $I(n,\eps'')$. By Lemma~\ref{l1}, for $n$ sufficiently large, the intersection of two of these intervals is reduced to the element $t$ and $t$ is an extremity of each of these intervals. This is clearly impossible as each of these intervals is of non-zero length.
\end {proof}

An immediate consequence of these statements is the

\subsection*{Proof of Proposition~\ref{un}}

\begin{proof}
Let $X$ be the set of infinite words $\eps$ for which the sequence $(u_{n,\eps})_n$ has no limit. We have to prove that $X$ is at most countable.

For $\eps$ in $X$, Lemma~\ref{l0} implies that $I(\eps)$ is not reduced to a single element. This allows us to choose a rational number $r(\eps)$ in
$I^o(\eps)$. Lemma~\ref{l1} then ensures that, for $\eps$ and $\eps'$ two distinct elements of $X$, we have $r(\eps)\not=r(\eps')$; $r$ is therefore an injective mapping from $X$ to $\Q$ which proves the proposition.
\end{proof}

\begin{prop}\label{A}
For $a>0$ the following statements are equivalent:

(i) \ Any element of $\r$ is representable.

(ii) \ For any infinite word $\eps$, the interval $I(\eps)$ is reduced to a single element.

(iii) \ $a$ is suitable.
\end{prop}

\begin{proof}
The implication $(i)\Rightarrow(ii)$ results from the remark following Lemma~\ref{l1'}.
Under the hypothesis $(ii)$, Lemma~\ref{l0} shows that any word $\eps$ represents an element of $\r$; the conjunction of Lemmas $\ref{l2}$ and $\ref{l0}$ ensures that every element of $\r$ is representable. This proves $(ii)\Rightarrow(iii)$.
The implication $(iii)\Rightarrow(i)$ follows from the definition of suitable real numbers.
\end{proof}
\section{The case $a\le1/e$}

We now assume $0<a\le1/e$. Let us introduce the function $g$ defined by $g(x)=xe^{-ax}$. Its derivative is given by $g'(x)=e^{-ax}(1-ax)$. Consequently, when $x$ varies from $-\infty$ to $+\infty$, the function increases from $-\infty$ to a maximum reached for $x=1/a$ then decreases to $0$. This maximum is $1/ae \geq 1$. Hence, there exist two real numbers $m$ and $M$ (possibly equal if $a= 1/e$) such that $m\le1/a\le M$ and
$g(m)=g(M)=1$. We then have $e^{am}=m$ and $e^{aM}=M$, so that $m$ and $M$ are fixed points for $f_+$.

\smallskip
Let us start with the case of the infinite word $\delta=++++\cdots$ formed only with the sign $+$ . We have $I(n,\delta)=[x_{n,\delta},+\infty]$
($+\infty$ is a fixed point of $f_+$) with $x_{0,\delta}=-\infty$ and $x_{n+1,\delta}=\exp(a\,x_{ n,\delta})$. In particular $x_{1,\delta}=0$,
$x_{2,\delta}=1$ and therefore, for $n\ge2$, $ x_{n,\delta}=u_{n-2,\delta}$. We then see that the $x_{n,\delta}$ and the $u_{n,\delta}$ tend towards the first fixed point of $f_+$ that is to say $m$.
In conclusion, $I(\delta)=[m,+\infty]$ and the word $\delta$ represents $m$.

For any finite word $\gamma$, let $X_\gamma$ be the image of the interval $]m,+\infty]$ by $f_\gamma$ and let $X$ be the union of all $X_\gamma$.

We have $]m,+\infty]\subset [m,+\infty]=I(\delta)$, hence $X_\gamma=f_\gamma(\,]m,+\infty])\subset f_\gamma(I(\delta))=I(\gamma\cat\delta)$.
For $\gamma'$ finite word $f_{\gamma'}(X_\gamma)=f_{\gamma'}\circ f_\gamma(]m,+\infty])=f_{\gamma'\cat\gamma}(\,]m,+\infty])=X_{\gamma'\cat\gamma}$.
It follows that $f_{\gamma'}(X)$ is contained in $X$.
In particular we have $\,]m,+\infty]=X_\emptyset\subset X$, $f_-(]m,+\infty])=[-\infty,-m[\subset X$ and $ f_+(-\infty)=0\in X$.

\begin{lem}\label{l3}
Any interval of nonzero length meets $X$ (i.e. $X$ is a dense subset of $\r$).
\end{lem}

\begin{proof}
Consider the nonempty intervals disjoint from $X$. By replacing such an interval by the union of the intervals that contain it but do not meet $X$, we obtain maximal intervals disjoint from $X$.
These intervals are pairwise disjoint since the union of two non-disjoint intervals is still an interval; these are the connected components of the complement of $X$.
As these intervals meet neither $[-\infty,-m[$ nor $]m,+\infty]$ they are contained in the interval $[-m,m]$ of length $2m$.
Since they are pairwise disjoint, for any $\eta>0$, there is only a finite number of them of length greater than $\eta$. It follows that there is one of maximum length, which we denote by $J$. Since $J$ does not meet $X$ it does not contain $0$ and therefore $J$ stays on the same side of $0$. Consequently, $J$ is the image of an interval $K$ by $f_+$ or $f_-$ according to whether $J$ remains positive or negative. Moreover, since $f_+(X)$ and $f_-(X)$ are contained in $X$, $K$ does not meet $X$. Consequently, the length of $K$ does not exceed that  of $J$ and $K\subset[-m,m]$. Now, on $]-m,m[$ we have $f_+'(x)=a\,e^{ax}<a\,e^{ am}=am\le 1$ (recall $m\le1/a$) and $f_-'(x)=-f_+'(x)$. It follows that, on $]-m,m[$, we have $|f'_+(x)|<1$ and $|f'_-(x)|<1$. The mean value theorem then ensures that the length of $K$ is zero or strictly greater than that of its image $J$. The last option being excluded, we conclude that $K$ is of zero length, i.e., reduced to a singleton. It is therefore the same for $J$ and for all the intervals considered since $J$ has a maximum length.
The lemma follows immediately.
\end{proof}

We can now give a

\subsection*{Proof of Proposition~\ref{deux}}

\begin{proof}
We first prove that any infinite word $\eps$ represents an element of $\r$.
It is obvious if $I(\eps)$ is reduced to a singleton by Lemma~\ref{l0}; therefore suppose $I(\eps)$ of non-zero length. Its interior $I^o(\eps)$ meets $X$, according to  Lemma~\ref{l3}. So there exists a finite word $\gamma$ such that $I^o(\eps)\cap X_\gamma\not=\emptyset$. As $X_\gamma\subset I(\gamma\cat\delta)$ we have \textit {a fortiori} $I^o(\eps)\cap I(\gamma\cat\delta)\not=\emptyset$ and Lemma~\ref{l1} ensures $\eps=\gamma\cat\delta$. As we have seen that $\delta$ represents $m$, it follows that $\eps$ represents $f_\gamma(m)$.
Thus, in all cases, $\eps$ represents an element of $\r$.
\medskip

We next prove that the set of representable elements of $\r$ is a
closed set with empty interior.  The proof will be done in several
steps, starting by proving that the set of representable elements of
$\r$ is the complement of $X$.

\paragraph{1. If $t\not\in X$ then $t$ is representable.}
By Lemma~\ref{l2} we can find $\eps$ such that $t\in I(\eps)$.
If $I(\eps)$ is reduced to a singleton, whence $I(\eps)=\{t\}$, then $t$ is represented by $\eps$ according to Lemma~\ref{l0}. Otherwise,
$I(\eps)$ meets $X$ by Lemma~\ref{l3}. In other words $I^o(\eps)$ meets $I(\gamma\cat\delta)$ for a finite word $\gamma$.
According to Lemma~\ref{l1}, this implies $\eps=\gamma\cat\delta$ and therefore $t\in I(\eps)=f_\gamma([m,+\infty])$. But then $t\not\in X$ implies $t=f_\gamma(m)$ and $t$ is represented by the word $\gamma\cat\delta$.

\paragraph{2. If $t\in X$ then $t$ is not representable.}
We therefore suppose $t=f_\gamma(u)$ for $\gamma$ a finite word and $u\in]m,+\infty]$. We show that $t$ is not represented by any infinite word. The proof will be done by induction on the length of $\gamma$.

If $\gamma$ is the empty word, then $f_\gamma$ is the identical map and $t=u\in]m,+\infty]$. Now,
$f_+([-\infty,m])=[0,m]\subset[-\infty,m]$ and $f_-([-\infty,m])=[-m,0]\subset [-\infty,m]$. In other words, the image of any element of $[-\infty,m]$ by
$f_+$ or $f_-$ is still in $[-\infty,m]$. As $[-\infty,m]$ contains 1 we conclude that, for any infinite word $\eps$ and any $n$, we have
$u_{n,\eps}=f_{n,\eps}(1)\in[-\infty,m]$. Passing to the limit we deduce that if $\eps$ represents an element $v$ we still have $v\in[-\infty,m]$. Hence no infinite word can represent $t$.

Suppose now that $\gamma$ has length $k\ge1$. Then $\gamma=\gamma_1\cat\gamma'$ for some $\gamma'$ of length $k-1$, and
$t=f_\gamma(u)=f_{\gamma_1}\big(f_{\gamma'}(u)\big)$.
If $t$ is represented by an infinite word $\eps=\eps_1\cat\eps'$, then for any integer $n$,
$u_{n+1,\eps}=f_{n+1,\eps}(1)=f_{\eps_1}\big(f_{n,\eps'}(1)\big)= f_{\eps_1}(u_{n,\eps'})$. It follows that $u_{n,\eps'}$ has a limit $v$ such that
$f_{\eps_1}(v)=t$. Thus $\eps'$ represents $v$.

 If $t>0$ (resp. $t<0$) we obviously have $\gamma_1=+=\eps_1$ (resp. $\gamma_1=-=\eps_1)$. If $t=0$, whether $\eps_1$ is the sign $+$ or the sign $-$, we have
$v=-\infty$ and therefore, even if it means modifying the sign $\eps_1$, we can always assume $\gamma_1=\eps_1$. We then have
$f_{\eps_1}(v)=t=f_{\gamma_1}\big(f_{\gamma'}(u)\big)$ and, since $\eps_1=\gamma_1$, we have $f_{\gamma'}(u)=v$ and therefore $v\in X$. As $\gamma'$ has length $k-1$ and $v$ is represented by the word $\eps'$ we obtain a contradiction with the induction hypothesis.

The set of representable $t$ is therefore the complement of $X$.

\smallskip
To complete the proof of Proposition~\ref{deux}, passing to the complement, it remains to prove that $X$ is an open and dense subset of $\r$.
The density of $X$ has been proved in Lemma~\ref{l3}, so let's prove:

\paragraph{3. $X$ is open in $\r$.}
The proof goes through another description of $X$. More precisely let us denote by $Y$ the union of the intervals $]m,+\infty]$, $[-\infty,-m[$ and
$f_\gamma(]-1/m,1/m[)$ for $\gamma$ finite word; we are going to prove $X=Y$.

Indeed $]-1/m,1/m[$ is the union of $]-1/m,0]=f_-([-\infty,-m[)$ and $[0,1/m [=f_+([-\infty,-m[)$.
As $[-\infty,-m[=f_-(]m,+\infty])$ and $]m,+\infty]=f_\emptyset(]m,+\infty])$, we conclude that $Y$ is the union of the images of $]m,+\infty]$ by $f_\emptyset$, $f_-$, $f_{\gamma\cat+\cat-}$ and $f_{\gamma\cat-\cat-}$ for $\gamma$ a finite word. These intervals are, by definition, contained in $X$, hence $Y\subset X$.

Moreover, noting that $f_+(]m,+\infty])=]m,+\infty]$, this description of $Y$ proves $f_+(Y)\subset Y$ and $f_-( Y)\subset Y$. It follows that $f_\gamma(Y)\subset Y$ for any finite word $\gamma$, and therefore $f_\gamma(]m,+\infty])\subset Y$, i.e. $X\subset Y$.

  The equality of $X$ and $Y$ follows.

\smallskip
Then the $f_\gamma(]-1/m,1/m[)$ are open intervals therefore open subsets of $\r$. And, although they are not open intervals, $]m,+\infty]$ and $[-\infty,-m[$ are open for the topology of $\r$ (as their complements are closed intervals). By union we deduce that $Y$ (and therefore $X$) is an open set of $\r$. This completes the proof.
\end{proof}

\section{The case $a>e$}

As in the previous case we will be interested in a fixed point not of $f_+$ but of $f_-$. We therefore consider the function $h$ defined by
$h(x)=x-f_-(x)=x+e^{ax}$. This function is strictly increasing from $-\infty$ to $+\infty$, so it vanishes at a single point $m$ (not the point $m$ of the previous section). We then have
$m=-e^{am}=f_-(m)$; $m$ is therefore the unique fixed point of $f_-$.

Also $f_-'(m)=-ae^{am}=am$ ; let us prove $f_-'(m)<-1$, that is to say $am<-1$, i.e. $m<-1/a$. This amounts to proving $h(-1/a)>0$, i.e.
$-1/a+1/e>0$, or equivalently $a>e$, which is true by hypothesis. Thus $m$ is a repulsive fixed point for $f_-$.

Consider then the infinite word $\delta'$ formed only of signs $-$. We therefore have $u_{n+1,\delta'}=f_-(u_{n,\delta'})$, consequently the only possible limit for the sequence of $u_{n,\delta'}$ is the only fixed point of $f_-$ which is $m$. However, since $m$ is repulsive, this requires the sequence of $u_{n,\delta'}$ to be stationary, which is in contradiction with $u_{0,\delta'}=1\not=m$ and $f_-$ is injective.

Thus $\delta'$ is an example of an infinite word for which the sequence of $u_{n,\delta'}$ diverges; hence, $\delta'$ does not represent any element of $\r$.

Moreover  Lemma~\ref{l0} ensures that $I(\delta')$ is not reduced to a point. It follows that $I^o(\delta')$ is not empty and, as $\delta'$ does not represent any element of $\r$,  Lemma~\ref{l1'} shows that no elements of $I^o(\delta')$ can be represented.

As  Proposition~\ref{un} ensures the existence of sequences $\eps$ for which the sequence of $u_{n,\eps}$ is convergent and therefore the existence of representable elements, these last two remarks prove Proposition~\ref{trois}.

\section{Looking for suitable elements}\label{suitable}

According to the above, these elements are to be searched in the interval $]1/e,e]$. Moreover  Proposition~\ref{A} says that these are the $a$ for which all the $I(\eps)$ are reduced to a single element. The fundamental tool is then the following lemma.

\begin{lem}\label{l6}
Suppose there is a function $\f$ from $\R$ to $\R$ such that
\begin{itemize}
\item[1.] $\f$ is continuous, even and strictly positive;
\item[2.]$\displaystyle\int_{-\infty}^{+\infty}\frac{dt}{\f(t)} < +\infty$;
\item[3.]$\forall x\in\R\quad \f(e^{ax})>ae^{ax}\f(x)$\quad except for a finite number of $x$ (for which there is a tie).
\end{itemize}
Then, for any infinite word $\eps$, the interval $I(\eps)$ is reduced to a single element.
\end{lem}

\begin{proof}
It is therefore a matter of proving that the intervals $I(\eps)$ are of zero length. For this we will argue as in the proof of Lemma~\ref{l3} but modifying the distance on $\r$ so that the maps $f_+$ and $f_-$ become contracting on all $\r$.

More precisely, for $I$ an interval with endpoints $x\leq y$ we will replace the length of $I$ by the measure $m(I)=\displaystyle \int_x^y\frac{dt}{\f( t)}$
(with $m(\emptyset)=0$). Assumption (2) ensures that $\r$ is of finite measure. The parity of $\f$ ensures that an interval $I$ and its symmetric $-I$ have the same measure.

The key point is then that hypothesis (3) ensures that the image by $f_+$ or $f_-$ of an interval $I$ has a measure less than that of $I$ with strict inequality if $I$ has nonempty interior.

This is obvious if $x=y$, i.e. if $I$ is a singleton. Otherwise, $x<y$ and the image of $I$ by $f_+$ is an interval $J$ with endpoints $e^{ax}$ and $e^{ay}$ and the image of $I$ by $f_-$ is the symmetric $-J$. We then have:

\[m\big(f_+(I)\big)=m\big(f_-(I)\big)=\int_{e^{ax}}^{e^{ay}}\frac{dt}{\f(t)}=\int_x^y\frac{ae^{au}du}{\f(e^{au})}
  \]
by making the change of variable $t=e^{au}$.

Assumption (3) allows us to strictly lower bound the last integral by 
$\displaystyle \int_x^y\frac{du}{\f(u)}=m(I)$.

Consider then all the open intervals $I^o(\eps)$. They are pairwise disjoint according to  Lemma~\ref{l1}. It follows that if we take a finite number of them, the sum of their measures is bounded by $m(\r)$ which is finite. In particular, for every $\eta>0$, there is only a finite number of them of measure greater than $\eta$ and, consequently,  (at least) one of these intervals has maximum measure.

Let $\eps'$ be an infinite word such that $I^o(\eps')$ has maximum measure. By writing $\eps'$ as the concatenation of the first sign $\eps'_1$ and the word
$\eps''$ formed by the other signs we obtain $I^o(\eps')=f_{\eps'_1}\big(I^o(\eps'')\big)$. The previous key point then ensures that $I^o(\eps'')$ has measure zero or strictly greater than that of $I^o(\eps')$. The last case is excluded because of the maximal nature of the measure of $I^o(\eps')$. It follows that $I^o(\eps'')$ and therefore $I^o(\eps')$ have measure zero, which implies that all $I^o(\eps)$ have measure zero and therefore all
$I(\eps)$ are reduced to a single element.
\end{proof}

Our goal will then be to exhibit explicit functions $\f$ satisfying points 1 and 2 in the lemma and to find $a$ such that point 3 holds, to conclude that these $a$ are suitable.

In fact, for point 3, we will only write the proof of

$(3')\ \forall x\in\R\quad \f(e^{ax})\ge ae^{ax}\f(x)$

the fact that equality can only occur for a finite number of values of $x$ being easy to establish.

\medskip\noindent
{\bf Example 1.}\quad $\f(t)=1+\l t^2$ \ , $\l>0$.

The validity of points 1 and 2 is obvious, let's look for couples $(\l,a)$ for which point 3' holds (although possible we will not try to describe them all but only to show enough of them to contribute with example 2 to the proof of  Proposition~\ref{quatre}).

It is therefore a matter of ensuring, for any real number $x$, the inequality $ 1+\l e^{2ax}\ge a e^{ax}(1+\l x^2)$. Multiplying by $e^{-ax}$, this is equivalent to
 $e^{-ax}+\l e^{ax}\ge a+\l a x^2$.

By introducing the variable $y=ax$, it is therefore a matter of ensuring that the function $F_{\l,a}$ defined by
$$F_{\l,a}(y)=e^{-y}+\l e^y-a-\l y^2/a$$
remains positive on all $\R$.

\smallskip\noindent
{\bf A simple case.} \quad $\l=1$ , $a=1$

The question is whether the quantity  $e^{-y}+e^y-1-y^2$ remains positive. By using the series expansion of $e^y$ and $e^{-y}$ we immediately see that it remains greater than $1$.

We have therefore proved that $1$ is a suitable real number; for the initial problem the answers to $Q_1$ and $Q_2$ are therefore always positive.

\smallskip\noindent
{\bf General case.}

In order to ensure the positivity of $F_{\l,a}$ we will require:

\noindent (a): $F_{\l,a}$ is convex.
 
There is a real number $t$ for which

\noindent (b): $F_{\l,a}(t)=0$ \ and \ $F'_{\l,a}(t)=0$.

These conditions will prove that $F_{\l,a}$ reaches a minimum at the point $t$ which is zero and therefore $F_{\l,a}$ is positive.

Condition (a) is written \ $F''_{\l,a}$ positive, or equivalently \ $e^{-y}+\l e^y-2\l/a\ge0$.

The quantity $e^{-y}+\l e^y-2\l/a$ reaches a minimum when $e^{-y}=\l e^y$, i.e. $e^{-y }=\sqrt\l$. This minimum is worth $2\sqrt\l-2\l/a$, so we see that condition (a) is equivalent to:
\[{\rm (a')} : \l\le a^2.
  \]

Condition (b) writes:
\[e^{-t}+\l e^t=a+\l t^2/a \quad {\rm and} \quad -e^{-t}+\l e^t=2\l t/a.
  \]

We then note that this condition is verified if we take $t$ in $]-2,2[$ and
\[a=(2+t)e^{-t} \quad {\rm and} \quad \l=\dfrac{(2+t)e^{-2t}}{2-t}
  \]

In order to satisfy condition (a'), one must also require:
$$\dfrac{(2+t)e^{-2t}}{2-t}\le(2+t)^2e^{-2t}$$
i.e. \ $1\le 4-t^2$ \ or $-\sqrt3\le t\le\sqrt3$ .

Hence, for $t\in[-\sqrt3,+\sqrt3]$, the real number $a=(2+t)e^{-t}$ is suitable.

When $t$ varies from $-\sqrt3$ to $+\sqrt3$ the quantity $(2+t)e^{-t}$ increases from $(2-\sqrt3) e^{\sqrt3}$ until at $e$ (reached for $t=-1$) then decreases to
$(2+\sqrt3)e^{-\sqrt3}$.

We have
\[(2-\sqrt 3) e^{\sqrt 3} \sim 1.51\dots \quad{\rm and} \quad (2+\sqrt3) e^{-\sqrt 3} \sim 0.66\dots
  \]

In conclusion all the real numbers of $[(2+\sqrt3)e^{-\sqrt3},e]$ are suitable.
We will retain in particular that all the real numbers of $[1,e]$ are suitable.

\medskip\noindent
{\bf Remark.} For the sake of simplicity we have limited the study to the case $F_{\l,a}$ convex. This is obviously not necessary. We can also prove, by keeping the same values for $\l$ and $a$ as functions of $t$, that $F_{\l,a}$ remains positive for $|t|\le 2\tanh |t|$. This proves that the elements of $[M,e]$ are suitable with $M\sim 0.577\dots$.
\medskip\noindent

{\bf Example 2.}\quad $\f(t)=\max\big(1,(|at|)^\nu\big)$ \ , $\nu>1$,\ $a>0$ .

The validity of points 1 and 2 is obvious, let us look for pairs $(\nu,a)$ for which point 3' holds.

It is therefore a question of ensuring \quad $\max(1,a^\nu e^{\nu ax})\ge a e^{ax}\max\big(1,(a|x|)^\nu\big)$\quad for any real number $x$.
  By introducing the variable $y=ax$ we are reduced to proving \quad$\max(1,a^\nu e^{\nu y})\ge a e^y\max(1,|y|)^\nu)$ .

This reduction is the conjunction of the two relations:
$$\max(1,a^\nu e^{\nu y})\ge ae^y \quad {\rm and} \quad \max(1,a^\nu e^{\nu y}) \ge ae^y|y|^\nu\ .$$

As the first is obviously always realized (since $\nu>1$), it suffices to prove the second.
To do this, it suffices to require:

{\bf 1.}\quad $1\ge ae^y|y|^\nu$\quad if $y<0$ \quad and \quad
{\bf 2.}\quad $a^\nu e^{\nu y}\ge ae^y|y|^\nu$\quad if $y>0$.

\smallskip
 Condition {\bf1} translates to: \quad $a\le e^{-y}(-y)^{-\nu}$ \quad if $y<0$.

 Condition {\bf2} translates to: \quad $a^{\nu-1}\ge e^{(1-\nu)y}y^\nu$ \quad if $y>0$.

The study of the derivative of $e^{-y}(-y)^{-\nu}$ for $y<0$ shows that this quantity reaches a strict minimum for $y=-\nu$ which equals $ (e/\nu)^\nu$.  Condition {\bf1} is therefore translated by $a\le(e/\nu)^\nu$ or equivalently $\nu a^{1/\nu}\le e$.

The study of the derivative of $e^{(1-\nu)y}y^\nu$ for $y>0$ shows that this quantity reaches a strict maximum for $y=\frac\nu{\nu- 1}$ which is
$e^{-\nu}\big(\frac\nu{\nu-1}\big)^\nu$. By setting $\nu'=\frac\nu{\nu-1}$, the condition {\bf 2} therefore results in $a\ge (\nu'/e)^{\nu'}$ or equivalently
$\nu'a^{-1/\nu'}\le e$.

Consider the special case where $\nu=1+1/a$, so we have $\nu'=1+a$. We then have $\nu a^{1/\nu}=\nu'a^{-1/\nu'}= f(a)$ where we have defined the function $f$ by
$f(x)=(x+1) x^{-1/x+1}$.

It follows that the conditions {\bf 1} and {\bf2} are fulfilled as soon as $f(a)\le e$.

The study of the logarithm of $f$ shows that between $0$ and $1$ $f$ decreases strictly from $+\infty$ to $2$ and that between $1$ and $+\infty$ $f$ increases strictly from $2$ to $+\infty$. As $2<e$, we conclude that the equation $f(x)=e$ has a unique solution $A$ between $0$ and $1$ and a unique solution $B$ between $1$ and $+\infty$. Furthermore the condition $f(a)\le e$ is equivalent to $a\in[A,B]$. The elements of $[A,B]$ are therefore suitable real numbers.

\smallskip
Therefore, since $1\in[A,B]$ and, as we have seen, all the real numbers of $[1,e]$ are suitable, we conclude that all the real numbers of $[A,e]$ are suitable.

This completes the proof of Proposition \ref{quatre}.

\medskip\noindent
{\bf Remarks.} For the sake of simplicity, we have limited the study to the case $\nu=1+1/a$. This is obviously not necessary but one can show that, even if one does not make this restriction, the method does not provide other suitable elements.

We can verify that we have $f(1/x)=f(x)$. It follows that we have $B=1/A$.

We have \quad
$A\sim 0.3942\dots$ \quad and \quad $B\sim 2.5367\dots$ .

\section*{Acknowledgements.}

We would like to thank all the colleagues and friends who have enabled us to improve this text through their remarks and comments, and especially Francis Lazarus who also provided the English translation of our initial version written in French.

\bibliographystyle{unsrt}
\bibliography{Lazarus}

\begin{thebibliography}{1}

\bibitem{EUL}
L~Euler.
\newblock De formulis exponentialibus replicatis.
\newblock {\em Euler Archive - All Works.}, 489:38--60, 1778.

\bibitem{BAK1}
I.~N. Baker and P.~J. Rippon.
\newblock Convergence of infinite exponentials.
\newblock {\em Ann. Acad. Sci. Fenn. Ser. A. I. Math.}, 8:179--186, 1983.

\bibitem{BAK2}
I.~N. Baker and P.~J. Rippon.
\newblock Iteration of exponential functions.
\newblock {\em Ann. Acad. Sci. Fenn. Ser. A. I. Math.}, 9:49--77, 1984.

\bibitem{MIS}
M.~Misiurewicz.
\newblock On iterates of $e^z$.
\newblock {\em Ergodic Theory and Dynamical Systems}, 1:103--106, 1981.

\bibitem{SHELL}
D.~L. Shell.
\newblock On the convergence of infinite exponentials.
\newblock {\em Proc. Amer. Math. Soc.}, 13:678--681, 1962.

\bibitem{THRON}
W.J Thron.
\newblock Convergence of infinite exponentials with complex elements.
\newblock {\em Proc. Amer. Math. Soc.}, 8:1040--1043, 1957.

\end{thebibliography}
\end{document}